\documentclass[a4paper,12pt]{amsart}
\usepackage{amsmath,amsthm,amsfonts,amssymb,comment}
\usepackage{multicol}
\usepackage{enumerate}
\usepackage[usenames,dvipsnames]{color}
\usepackage[toc,page]{appendix}
\usepackage{amssymb,latexsym,url}
\usepackage[pdftex]{graphicx}
\usepackage{tikz}
\usetikzlibrary{3d,calc}
\usepackage{subcaption}
\usepackage{overpic}
\usepackage[unicode]{hyperref}
\usepackage{hypbmsec}
\usepackage{booktabs}

\newtheorem{thm}{Theorem}[section]
\newtheorem{cor}[thm]{Corollary}
\newtheorem{lemma}[thm]{Lemma}
\newtheorem{prop}[thm]{Proposition}

\newtheorem{conj}[thm]{Conjecture}
\theoremstyle{definition}
\newtheorem{remark}[thm]{Remark}

\numberwithin{equation}{section}
\numberwithin{figure}{section}

\def\al{\alpha}
\def\be{\beta}

\def\de{\delta}

\def\R{\mathbb{R}}
\def\C{\mathbb{C}}
\def\N{\mathbb{N}}

\def\Id{\boldsymbol{1}}

\def\lc{\mathrm{lc}}

\def\supp{\operatorname{supp}}

\let\<\langle
\let\>\rangle

\title[A partial-sum deformation for a family of orthogonal polynomials]{A partial-sum deformation\\for a family of orthogonal polynomials}
\author{Erik Koelink}
\address{IMAPP, Radboud Universiteit, Nijmegen, The Netherlands}
\email{e.koelink@math.ru.nl}
\author{Pablo Rom\'an}
\address{FaMAF-CIEM, Universidad Nacional de C\'ordoba, Argentina}
\email{pablo.roman@unc.edu.ar}
\author{Wadim Zudilin}
\address{IMAPP, Radboud Universiteit, Nijmegen, The Netherlands}
\email{w.zudilin@math.ru.nl}

\date{\today}

\begin{document}

\dedicatory{Dedicated to Tom Koornwinder on the occasion of his 80th birthday}

\begin{abstract}
There are several questions one may ask about polynomials $q_m(x)=q_m(x;t)=\sum_{n=0}^mt^np_n(x)$ attached to a family of orthogonal polynomials $\{p_n(x)\}_{n\ge0}$.
In this note we draw attention to the naturalness of this partial-sum deformation and related beautiful structures.
In particular, we investigate the location and distribution of zeros of $q_m(x;t)$ in the case of varying real parameter $t$.
\end{abstract}

\maketitle


\section{Introduction}\label{sec:intro}

In the study of orthogonal polynomials and special functions a very useful tool is a generating function for the special functions at hand. Generating functions can be used to obtain e.g.\ asympotic results for the orthogonal polynomials or explicit identities for the special functions at hand. This paper is dedicated to Tom Koornwinder, who has championed use of generating functions to obtain orthogonality relations for $q$-series and orthogonal polynomials. In \cite{KoorS} Koornwinder and Swarttouw obtain $q$-analogs of the Hankel transform using generating functions. In \cite{CaglK} Cagliero and Koornwinder use generating functions to obtain inverses of infinite lower triangular matrices involving Jacobi polynomials. Similar matrices play an important role in weight functions for matrix-valued orthogonal polynomials, see \cite{KoeldlRR} for an example.

In our study  \cite{KRZ24} of zeros of polynomials related to matrix-valued orthogonal polynomials as introduced in \cite{KoeldlRR} we encountered a particular deformation method that was producing interesting objects worth of their own investigation. Essentially, we are studying partial sums of 
generating functions for orthogonal polynomials. These partial
sums can be viewed as deformations of the orthogonal polynomials.
In this paper we study this deformation in more detail. 
How much structure remains when a family of orthogonal polynomials is deformed in a `reasonably natural' way? Do new structure(s) appear?
Of course, the answer to these ill-posed questions depend on many factors but mainly on the \emph{family} itself and the \emph{way} it is deformed.
At the same time the existing literature revealed to us only a few and very particular examples, so we attempted to understand the new structures and phenomena arising ourselves.
The goal of this note is to present our findings publicly, in a conventional text format (in particular, leaving some amazing animation out).

The setup for our deformation is quite simple.
We start with a family $\{p_n(x)\}_{n\ge0}$ of polynomials which are
orthogonal with respect to nontrivial probability measure supported on the real line.
We introduce a \emph{positive real} parameter $t$%
\footnote{One can also deal with negative real $t$ along the lines. We omit related considerations for simplicity of our presentation.}
and assign to it another set of polynomials
\[
q_m(x;t) = \sum_{n=0}^m t^n p_n(x) \quad\text{for}\; m=0,1,2,\dots,
\]
which can be thought as truncations (or partial sums) of the generating function
\[
q(x;t) = q_\infty(x;t) = \sum_{n=0}^\infty t^n p_n(x)
\]
for the original family.
It is well known that the \emph{zeros} of polynomials $p_n(x)$ are real, simple and located
in the interior of the convex hull of the support of the measure;
furthermore, the zeros of two consecutive polynomials $p_{n-1}(x)$ and $p_n(x)$ interlace: any interval connecting two neighbouring zeros of the latter contains exactly one zero of the former.
An experienced analyst will suspect that these basic properties of orthogonal polynomials are inherited by their newly created mates when $t$~is chosen to be sufficiently large.
Indeed, the latter setting means that the term $t^mp_m(x)$ dominates in the expression for $q_m(x;t)$\,---\,in what follows we convert this into a rigorous argument.
How large this $t$ should be? What is the shape of the range for $t$, for which the above properties (or some of them) are still valid?
Based on extended numerical experiment we expect (however are not able to prove, even for a particular choice of a family $p_n(x)$) this range to be always of simple shape $t>t_{\textrm{crit}}$ with the value of $t_{\textrm{crit}}$ depending on~$m$.
What is the asymptotics of such $t_{\textrm{crit}}(m)$ as $m\to\infty$?
What happens to the zeros of $q_m(x;t)$ when $t<t_{\textrm{crit}}(m)$?
What is their limiting distribution as $t\to0^+$?
We address these and related questions, mainly from a numerical perspective, to convince the reader in the existence of new interesting structures going far beyond the orthogonality.

Less surprisingly, perhaps, is that the three-term recurrence relations for orthogonal polynomials $p_n(x)$ reflect on their partial sums $q_m(x;t)$; we record the resulting \emph{four-term recurrence relation} in Section~\ref{sec:generalsetup} below.
We also demonstrate that the knowledge of the generating function $q(x;t)$ gives a simple access to the generating function of polynomials $q_m(x;t)$.
Section~\ref{sec:genzerospartialsums} focuses on the structure of zeros of partial sums $q_m(x;t)$ in generic situations; it culminates in Conjecture~\ref{conj:criticalvaluet} about the turning value of $t>0$ after which the zeros become not real-valued.
In Section \ref{sec:examples} we illustrate our theoretical findings and expectations on particular examples of partials sums of Hermite, Charlier and Lommel polynomials as diverse representatives of families of orthogonal polynomials.
In this section we draw some pictures but also (draw) attention to a limiting distribution of zeros of partial sums of Hermite polynomials, namely, its link with the Szeg\H{o} curve.
Our brief Section~\ref{sec:5} speculates on further connections with other deformations.

\subsection*{Acknowledgements}
We are truly thankful to Andrei Mart\'inez-Finkelshtein for discussions about the zeros of special polynomials.
We thank the anonymous referees for enthusiastic reports and valuable comments that we tried to address in the final version.

\section{General setup}\label{sec:generalsetup}


Let $\{p_n(x)\}_{n=0}^\infty$ be a set of polynomials such that $p_n$ has degree $n$ and there exists a three-term recurrence relation
\begin{equation}\label{eq:gen3termrecrelgenOP}
x\, p_n(x) = a_n\, p_{n+1}(x) + b_n\, p_n(x) 
+ c_n\, p_{n-1}(x),
\quad\text{where}\; n\ge0,
\end{equation}
with the initial values $p_{-1}(x)=0$, $p_0(x)=1$ and 
\emph{real} sequences $(a_n)_{n=0}^\infty$,  
$(b_n)_{n=0}^\infty$ and $(c_n)_{n=1}^\infty$. We assume $a_n\ne0$ for all $n\in \N$, so that the polynomials $p_n(x)$ are determined by \eqref{eq:gen3termrecrelgenOP} and the initial conditions. 
We are interested in the partial sums of the generating function $q(x;t)=\sum_{n=0}^\infty t^n \al_n p_n(x)$ and we assume that the normalisation factors satisfy $\al_n\in \R\setminus\{0\}$ for all $n\in \N$ and $\al_0=1$. We denote by 
\begin{equation}\label{eq:genpartialsumgenfunction}
q_m(x;t) = \sum_{n=0}^m t^n \, \al_n\, p_n(x)
\end{equation}
these partial sums. We assume $t\ne0$, otherwise 
$q_m(x;0)=1$ is trivial, and typically we assume $t>0$. Note that there is ambiguity in the normalisation, since we have general
constants in \eqref{eq:gen3termrecrelgenOP} and an arbitrary sequence $\al_n$ in the partial sum of the generating function. One natural choice is to make the polynomials monic, i.e.\ $a_n=1$ for all $n$, or, assuming $a_{n-1}c_n>0$ for all $n\geq 1$, to normalise $c_n=a_{n-1}$ and the polynomials being orthonormal with respect to a positive measure $\mu$ on the real line. 

For $k\in \N$ we also require the $k$-th associated polynomials 
$p_n^{(k)}$, which have degree $n$ and satisfy
\begin{equation}\label{eq:gen3termrecrelgenassociatedOP}
x\, p_n^{(k)}(x) = a_{n+k}\, p_{n+1}^{(k)}(x) + b_{n+k}\, p_n^{(k)}(x) 
+ c_{n+k}\, p_{n-1}^{(k)}(x),
\quad\text{where}\; n\ge0,
\end{equation}
with the initial values $p_{-1}^{(k)}(x)=0$, $p_0^{(k)}(x)=1$.

\begin{lemma}\label{lem:gen4termrecurforpartialsums}
The sequence $\{q_m\}_{m\in \N}$ of partial sums satisfies the  
four-term recursion:
\begin{align*}
&\frac{a_m}{\al_{m+1}} q_{m+1}(x;t) 
+ \frac{t\al_{m+1}(b_m-x)-a_m\al_m}{\al_m\al_{m+1}}
q_m(x;t) \\ &\qquad
+t \frac{t c_m\al_m-\al_{m-1}(b_m-x)}{\al_{m-1}\al_m} q_{m-1}(x;t) - \frac{t^2 c_m}{\al_{m-1}} q_{m-2}(x;t)=0,
\quad\text{where}\; m\ge0,
\end{align*}
with initial values $q_{-2}(x;t) = q_{-1}(x;t)=0$,
$q_0(x;t)=1$.
\end{lemma}

Note that the recurrence relation of Lemma \ref{lem:gen4termrecurforpartialsums} determines the polynomial $q_m$, since the coefficient of $q_{m+1}(x;t)$
is nonzero. 

\begin{proof} 
Write 
\begin{align*}
&\, 
A_m \bigl[ q_{m+1}(x;t)- q_{m}(x;t) \bigr] 
+ B_m \bigl[ q_{m}(x;t)- q_{m-1}(x;t) \bigr] 
+ C_m \bigl[ q_{m-1}(x;t)- q_{m-2}(x;t) \bigr] 
\\ &\qquad
= A_m t^{m+1} \al_{m+1} p_{m+1}(x) 
+ B_m t^{m} \al_{m} p_{m}(x) 
+ C_m t^{m-1} \al_{m-1} p_{m-1}(x);
\end{align*}
this vanishes if we take $A_m\al_{m+1} = a_m$,
$B_m \al_m = t(b_m-x)$, $C_m\al_{m-1} = t^2 c_m$ by 
\eqref{eq:gen3termrecrelgenOP}. 
Plugging these values for $A_m$, $B_m$, $C_m$ in
and rearranging the terms gives the result. 
\end{proof}

Write Lemma \ref{lem:gen4termrecurforpartialsums} as
\begin{align*}
&
\frac{a_m}{\al_{m+1}} q_{m+1}(x;t) 
+ \frac{t\al_{m+1}b_m-a_m\al_m}{\al_m\al_{m+1}}
q_m(x;t) \\ &\qquad
+ t \frac{t c_m\al_m-\al_{m-1}b_m}{\al_{m-1}\al_m} q_{m-1}(t,x) - \frac{t^2 c_m}{\al_{m-1}} q_{m-2}(x;t)= xt \frac{1}{\al_m} \bigl( q_m(x;t) - q_{m-1}(x;t)\bigr)
\end{align*}
and define the semi-infinite matrix $L = (L_{i,j})_{i,j\in \N}$ by 
\begin{equation}\label{eq:gendefL}
\begin{aligned}
L_{m,m+1} &= \frac{a_m}{\al_{m+1}} &\quad\text{for}\; m&\in \N,  \\
L_{m,m} &= \frac{t\al_{m+1}b_m-a_m\al_m}{\al_m\al_{m+1}} &\quad\text{for}\; m&\in \N, \\
L_{m,m-1} &= t \frac{t c_m\al_m-\al_{m-1}b_m}{\al_{m-1}\al_m} &\quad\text{for}\; m&\in \N_{\geq 1}, \\ 
L_{m,m-2} &= - \frac{t^2 c_m}{\al_{m-1}} &\quad\text{for}\; m&\in \N_{\geq 2}, 
\end{aligned}
\end{equation}
and $L_{i,j}=0$ otherwise. Similarly, define $M=(M_{i,j})_{i,j\in\N}$ by 
\begin{equation}\label{eq:gendefM}
\begin{aligned}
M_{m,m} &= \frac{1}{\al_m} &\quad\text{for}\; m&\in \N, \\
M_{m,m-1} &= -\frac{1}{\al_m} &\quad\text{for}\; m&\in \N_{\geq 1},
\end{aligned}
\end{equation}
and $M_{i,j}=0$ otherwise. Introduce the vector
\begin{equation*}
\boldsymbol{q}(x;t) = \begin{pmatrix} q_0(x;t) \\ q_1(x;t) \\ q_2(x;t) \\ \vdots \end{pmatrix};
\end{equation*}
then Lemma \ref{lem:gen4termrecurforpartialsums} is equivalent to the generalised eigenvalue equation
\begin{equation}\label{eq:genGFgeneigvaleq}
L(t)\, \boldsymbol{q}(x;t) = xt\, M\, \boldsymbol{q}(x;t).
\end{equation}
Observe that $L(t)$ depends on $t$, while $M$ is independent of $t$. 

Next we define the truncations $L_N(t)$, respectively $M_N$, as the 
$(N+1)\times (N+1)$-matrices, by keeping the first $N+1$ rows and columns of $L(t)$ and $M$. 

\begin{prop}\label{prop:gengenfunctionasdetofalmosttridiagmatrix}
The matrix $P_N(t) = t^{-1} L_N(t)M_N^{-1}$ is tridiagonal, except for its last row. Explicitly, 
\begin{gather*}
P_N(t)_{i,i-1} = t c_i, \qquad P_N(t)_{i,i} = b_i, \qquad 
P_N(t)_{i,i+1} = \frac{a_i}{t}, \\
P_N(t)_{N,j} =  -\frac{\al_j\, a_N}{t \al_{N+1}}, \quad 
P_N(t)_{N,N-1} = t c_N -\frac{\al_{N-1}\, a_N}{t \al_{N+1}}, \quad
P_N(t)_{N,N} = b_N -\frac{\al_{N}\, a_N}{t\al_{N+1}},
\end{gather*}
and then
\begin{align*}
q_{N+1}(x;t)
&= (-t)^{N+1} \al_{N+1} \left( \prod_{k=0}^N a_k\right) 
\, \det(P_N(t) -x\Id) \\
&= (-1)^{N+1} \left( \prod_{k=0}^{N} \frac{\al_{k+1}}{a_k} 
\right) \, \det(L_N(t)-xtM_N). 
\end{align*}
The geometric multiplicity of any eigenvalue $x$ of $P_N(t)$ is $1$.
\end{prop}

\begin{cor}\label{cor:prop:gengenfunctionasdetofalmosttridiagmatrix}
If $P_N(t)$ is semisimple, then the zeros of 
$x\mapsto q_{N+1}(x;t)$ are simple. If $x\mapsto q_{N+1}(x;t)$
has a zero of higher multiplicity at $x_0$ for the value $t=t_0$, then 
$P_N(t_0)$ has a higher-order Jordan block for the eigenvalue $x_0$. 
\end{cor}

\begin{proof}[Proof of Proposition~\ref{prop:gengenfunctionasdetofalmosttridiagmatrix}]
$M_N$ is invertible as a lower triangular matrix with
$\det(M_N)= \prod_{i=0}^N \al_i^{-1}$, and 
$tP_N(t) M_N = L_N(t)$ follows by a direct verification. 
Now Lemma \ref{lem:gen4termrecurforpartialsums} gives 
\begin{equation}\label{eq:genGFgeneigvalequptoN}
L_N(t) \, \boldsymbol{q}_N(x;t) + 
\frac{a_N}{\al_{N+1}}
\begin{pmatrix} 0 \\ \vdots \\ 0 \\ q_{N+1}(x;t) \end{pmatrix} = xt \, M_N\, \boldsymbol{q}_N(x;t), 
\qquad 
\boldsymbol{q}_N(x;t) = \begin{pmatrix} q_0(x;t)\\ \vdots  \\ q_{N-1}(x;t) \\ q_N(x;t) \end{pmatrix}.
\end{equation}
It follows that $x_0\in \C$ satisfies $q_{N+1}(x_0;t)=0$ if and only if 
$\det(L_N(t)-x_0t\, M_N)=0$, and the latter is equivalent to
if $\det(P_N(t) -x_0\Id)=0$. Since the matrix $P_N(t)$ is tridiagonal with non-zero upper diagonal apart from the last row, it follows that any eigenvector for the eigenvalue $x_0$ is completely determined by its first entry. So the geometric multiplicity is at most one.

Observe next that \eqref{eq:genGFgeneigvalequptoN} gives 
\[
\boldsymbol{p}_N(x;t) = M_N \boldsymbol{q}_N(x;t) = \begin{pmatrix} t^0p_0(x) \\ \vdots \\ t^Np_N(x) \end{pmatrix} 
\quad \Longrightarrow 
\quad (P_N(t) -x\Id) \boldsymbol{p}_N(x;t) = 
\frac{- a_N}{t \al_{N+1}}
\begin{pmatrix} 0 \\ \vdots \\ 0 \\ q_{N+1}(x;t) \end{pmatrix},
\]
so the zeros of $x\mapsto q_{N+1}(x;t)$ are the eigenvalues of $P_N(t)$. Then  
$x\mapsto \det(P_N(t) -x\Id)$ and $x\mapsto q_{N+1}(x;t)$ are polynomials of the same degree with the same zeros, so they differ by a constant. Considering the leading coefficient gives the result. 
\end{proof}

\begin{remark}\label{rmk:prop:gengenfunctionasdetofalmosttridiagmatrix}
Recalling that $L_N(t)-xtM_N$ has four non-trivial diagonals, we can,  by developing along the last row, obtain a four-term recursion in $N$ for $\det(L_N(t)-xtM_N)$ following 
 \cite[Thm.~1]{CahiDNN}. This four-term recursion can be matched
 with the four-term recursion of Lemma \ref{lem:gen4termrecurforpartialsums} after rescaling. Then a check for the initial values gives another proof for the first expression of 
 $q_{N+1}(x;t)$ in Proposition \ref{prop:gengenfunctionasdetofalmosttridiagmatrix}.
\end{remark}

\begin{cor}\label{cor:prop:genpartsumisdetofHessenberg} 
The polynomials $p_n$ can be obtained from $\det(L_N(t)-xtM_N(t))$ via  
\begin{align*}
p_n(x) &= \frac{(-1)^{N+1}}{\al_n\, n!}  \left( \prod_{k=0}^{N} \frac{\al_{k+1}}{a_k} 
\right) \, 
\frac{d^n}{dt^n }\Big\vert_{t=0} \det(L_N(t)-xtM_N) \quad\text{for}\; N> n \\
&= \frac{(-1)^{n}}{\al_n\, n!}  \left( \prod_{k=0}^{n-1} \frac{\al_{k+1}}{a_k} 
\right) \, 
\frac{d^n}{dt^n } \det(L_{n-1}(t)-xtM_{n-1}).
\end{align*}
\end{cor}

Next we look into possible zeros of higher multiplicity as in 
Corollary~\ref{cor:prop:gengenfunctionasdetofalmosttridiagmatrix}. For this we write $P_N(t) = J_N(t) + R_N(t)$, where $J_N(t)$ is the truncation of the 
Jacobi matrix $J(t)$ with $J(t)_{i,i}=b_i$, $J(t)_{i,i+1}=a_i/t$, 
$J(t)_{i-1,i}=c_it$ and $R_N(t)_{i,j}= \de_{i,N} \frac{-\al_j\, a_N}{t\, \al_{N+1}}$ is the matrix with nonzero entries only in the last row. 
Corollary~\ref{cor:prop:gengenfunctionasdetofalmosttridiagmatrix} states that we need to find nontrivial vectors $w=w(x;t)$ such that 
$(P_N(t_0)-x_0)w(x;t) = \boldsymbol{p}_N(x_0;t_0)$ with $q_{N+1}(x_0;t_0)=0$ in order to have a zero of higher multiplicity for $x\mapsto q_{N+1}(x;t_0)$ 
at $x_0$. 

\begin{lemma}\label{lem:partialinverseJacobi}
Define the strictly lower triangular matrix $B_N(x;t)$ by 
\[
B_N(x;t)_{i,j} = \begin{cases} 0 & \text{for}\; i\leq j, \\
\dfrac{t^{i-j}}{a_j} p_{i-j-1}^{(j+1)}(x) & \text{for}\; i\geq j+1, 
\end{cases} \qquad 0\leq i,j \leq N, 
\]
in terms of the associated polynomials, see eq.~\eqref{eq:gen3termrecrelgenassociatedOP}. 
Then $\bigl( (J_N(t)-x)B_N(x;t)\bigr)_{i,j} =\de_{i,j}$ for $i<N$ and 
$\bigl( (J_N(t)-x)B_N(x;t)\bigr)_{N,N} =0$ and 
\[
\bigl( (J_N(t)-x)B_N(x;t)\bigr)_{N,j} = -\frac{t^{N-j}a_N}{a_j} p_{N-j}^{(j+1)} \qquad\text{for}\; 0\leq j <N.
\]
\end{lemma}

\begin{proof} This is a straightforward calculation. Since 
$J_N(t)$ is tridiagonal and $B_N(x;t)$ is strictly lower triangular, the product is lower triangular. For $i=0$ the result follows easily.
For $i>0$ we can restrict to $j\leq i$, so that 
\begin{align*}
&\, \bigl( (J_N(t)-x)B_N(x;t)\bigr)_{i,j} = 
c_it B_N(x;t)_{i-1,j} + (b_i-x) B_N(x;t)_{i,j} + t^{-1} a_i B_N(x;t)_{i+1,j}
\end{align*}
which is $1$ in case $j=i$. In case $j<i$, it equals 
\begin{align*}
&
c_it \frac{t^{i-1-j}}{a_j} p_{i-j-2}^{(j+1)}(x) + (b_i-x) 
\frac{t^{i-j}}{a_j} p_{i-j-1}^{(j+1)}(x) + t^{-1} a_i 
\frac{t^{i+1-j}}{a_j} p_{i-j}^{(j+1)}(x) \\ &\qquad
= \frac{t^{i-j}}{a_j} 
\bigl( c_i  p_{i-j-2}^{(j+1)}(x) + (b_i-x) 
p_{i-j-1}^{(j+1)}(x) +  a_i p_{i-j}^{(j+1)}(x) 
\bigr) = 0
\end{align*}
by \eqref{eq:gen3termrecrelgenassociatedOP}. It remains to calculate the final row. The $(N,N)$-entry is zero, since $B_N(x;t)$ is strictly lower triangular and the expressions for other entries follow from \eqref{eq:gen3termrecrelgenassociatedOP}.
\end{proof}

Lemma \ref{lem:partialinverseJacobi} states that $w(x;t) = B_N(x;t)\boldsymbol{p}_N(x;t)$ is a potential generalised eigenvector. Explicitly, by applying 
Lemma~\ref{lem:partialinverseJacobi} to $\boldsymbol{p}_N(x;t)$ we obtain the following result.

\begin{cor}\label{cor:lem:partialinverseJacobi}
In the notation
\[
r_{N+1}(x;t) = -t^N \sum_{j=0}^N \frac{a_N}{a_j} p_{N-j}^{(j+1)}(x)\, p_j(x)
\]
we have 
\[
(J_N(t)-x)B_N(x;t) \boldsymbol{p}_N(x;t) = \boldsymbol{p}_N(x;t) + r_{N+1}(x;t) 
\begin{pmatrix} 0 \\ \vdots \\ 0 \\ 1 \end{pmatrix}.
\]
\end{cor}

Note that Lemma \ref{lem:partialinverseJacobi} and Corollary \ref{cor:lem:partialinverseJacobi} only make use of the three-term recurrence relation, so that we can take $t=1$ without loss of generality. 
In case $a_{n-1}c_n>0$ for all $n\geq 1$ we have that $J(1)$ is semisimple, hence the zeros of $p_{N+1}$, which form the spectrum of $J(1)$, do not coincide with zeros of $r_{N+1}(x;1)$; in other words, 
$p_{N+1}(x)$ and $\sum_{j=0}^N a_j p_{N-j}^{(j+1)}(x)\, p_j(x)$ have no common zeros. 
Therefore, for orthogonal polynomials Lemma \ref{lem:partialinverseJacobi} does not add any information. In the case of partial sums of generating functions, we need to switch from $J_N(t)$ to $P_N(t) = J_N(t) + R_N(t)$. For this we just need to modify with $R_N(t) B_N(x;t) \boldsymbol{p}_N(x;t)$, which only adds to the last row. 
This results in the following statement, which characterises multiple zeros of the polynomial $q_N(x;t_0)$ when $t_0$~is fixed.

\begin{prop}\label{cor2:lem:partialinverseJacobi} 
Let 
\[
s_{N+1}(x;t) = \frac{-a_N}{\al_{N+1}} 
\sum_{k=1}^N \frac{\al_k}{a_k} t^{k-1} \sum_{j=0}^{k-1} p_{k-1-j}^{(j+1)}(x)\, p_j(x)
\]
and $r_{N+1}(x;t)$ as in Corollary \ref{cor:lem:partialinverseJacobi}.
Assume $q_{N+1}(x_0;t_0)=0$.  Then $x_0$ is a zero of higher multiplicity
of $x\mapsto q_{N+1}(x;t_0)$ if and only if 
$r_{N+1}(x_0;t_0) + s_{N+1}(x_0;t_0) = 0$.
\end{prop}

Note that $r_{N+1}(x;t) + s_{N+1}(x;t)$ can be written as a kind of generating function for the polynomials $r_n(x;t)$ of lower degree; namely,
\begin{align*}
r_{N+1}(x;t) + s_{N+1}(x;t)
&= - a_N 
\sum_{i=1}^{N+1} t^{i-1} \frac{\al_i}{\al_{N+1}}
\sum_{j=0}^{i-1} p_{i-j-1}^{(j+1)}(x) \frac{p_j(x)}{a_j}
\\
&= \frac{a_N}{\al_{N+1}}
\sum_{i=1}^{N+1} \frac{\al_i}{a_{i-1}}\, r_i(x;t).
\end{align*}

One observation, which follows from a simple manipulation with sums, is 
that we have a generating function for the partial sums in case  the 
corresponding generating function of the polynomials is known.

\begin{lemma}\label{lem:genfun}
If $f(x;y)=\sum_{n=0}^\infty\al_np_n(x)y^n$, then the partial sums 
have the generating function 
\[
\sum_{m=0}^\infty q_m(x;t)y^m
=\frac{f(x;ty)}{1-y}.
\]
\end{lemma}

\begin{proof} The left-hand side equals 
\[
\sum_{m=0}^\infty y^m\sum_{n=0}^mt^n\al_np_n(x)
= \sum_{n=0}^\infty t^n\al_np_n(x) \sum_{m=n}^\infty y^m 
=\frac1{1-y}\sum_{n=0}^\infty\al_np_n(x)(ty)^n
=\frac{f(x;ty)}{1-y}. 
\]
If $\sum_{n=0}^\infty\al_np_n(x)y^n$ has radius of convergence
$|y|<R$, the generating function for the partial sums converges absolutely for $|yt|<R$. 
\end{proof}


\section{Zeros of partial sums}\label{sec:genzerospartialsums}


Recall the assumptions at the beginning of 
Section~\ref{sec:generalsetup}. From now on we further assume that polynomials satisfy the conditions of Favard's theorem (see, e.g., \cite{AndrAR}, \cite{Isma}), namely, that $a_{n-1}c_n>0$ for all $n\geq 1$. In particular, this implies that each polynomial $p_n$ from the orthogonal family has $n$ simple real zeros, denoted 
$x_1^{n}<x_2^{n}< \cdots <x_n^n$. Moreover, the zeros of 
$p_n$ and $p_{n+1}$ strictly interlace and the zeros are contained in the 
convex hull of the support $\supp(\mu)\subset \R$ of the orthogonality measure $\mu$ for the polynomials. 

By the realness assumptions, the zeros of $q_m(x;t)$ are either real or appear in complex conjugate pairs. 
For $N=0$ we have $q_0(x;t)=1$, which has no zeros. For $N=1$ we get 
$q_1(x;t) = 1+ \al_1 t \frac{x-b_0}{a_0}$ which has one real 
zero $b_0-\frac{a_0}{\al_1 t}$. For $N=2$ we get 
\begin{equation}
q_2(x;t) =1 + \al_1 t \frac{x-b_0}{a_0} + \al_2t^2 
\Bigl( \frac{(x-b_0)(x-b_1)}{a_0a_1} -\frac{c_1}{a_1}\Bigr)
\end{equation}
and viewing this as a quadratic polynomial in $x$, its discriminant is a polynomial in $t$ of degree $4$ with leading coefficient 
\[
\frac{\al_2^2}{a_0^2a_1^2}(b_0-b_1)^2 + 4 \frac{c_1 \, \al_1^2}{a_0 a_1^2} >0.
\]
Therefore, for $t$ sufficiently large there are two real simple zeros for $x\mapsto q_2(x;t)$. This realness remains true for general $N$.

\begin{prop}\label{prop:genrealzerosfortbig}
For $t\gg 0$ the partial sum $x\mapsto q_N(x;t)$ has real simple zeros. Moreover, for $t\gg 0$ the real simple zeros of $x\mapsto q_N(x;t)$ and the real simple zeros of 
$x\mapsto q_{N+1}(x;t)$ interlace. 
\end{prop}

The idea for the proof below comes from the entry in Mathoverflow \cite{MOuser}
due to an unknown user, who motivates it as a real-valued counterpart of Krasner's lemma from $p$-adic analysis. 

\begin{proof} Write
\begin{equation*}
q_N(x;t) = \al_N t^N p_N(x) + q_{N-1}(x;t), 
\end{equation*}
viewing $q_N(x;t)$ as a perturbation of $p_N(x)$, which has real and simple zeros. Put 
\begin{equation*}
m_N = \min_{1\leq i\ne j \leq N} |x^N_i -x^N_j | = 
\min_{1\leq i< N} |x^N_{i+1} -x^N_{i} | >0 
\end{equation*}
to be the minimum distance between the zeros of $p_N$. Notice that $m_N$ is independent of $t$. 
Write 
\begin{equation}
q_N(x;t) = t^N \,\al_N\, \lc(p_N) \, \prod_{j=1}^N (x-\be_j^N)
\end{equation}
where $\{\be_1^N,\dots, \be_N^N\}$ are the zeros of $x\mapsto q_N(x;t)$ (with multiplicity) and the leading coefficient of $q_N(x;t)$ is $t^N\al_N$ times the leading coefficient of $p_N$. Now $q_{N-1}(x;t)$ is a polynomial of 
degree $N-1$ in $t$, and we pick $t\gg 0$ such that 
\begin{equation}\label{eq:prop:genrealzerosfortbig1}
|q_{N-1}(x_i^N;t)| < |t|^N\, |\al_N|\, |\lc(p_N)| \Bigl( \frac{m_N}{2} \Bigr)^N \qquad\text{for all}\; i\in \{1,\dots, N\}.
\end{equation}
With this choice of $t$ we evaluate $q_N$ at a zero $x^N_i$ of $p_N$: 
\begin{align*}
t^N \,\al_N\, \lc(p_N) \, \prod_{j=1}^N (x^N_i-\be_j^N) = 
q_N(x^N_i;t) =  q_{N-1}(x^N_i;t),
\end{align*}
hence $\prod_{j=1}^N |x^N_i-\be_j^N| < (\frac12 m_N)^N$ by 
taking absolute values and using \eqref{eq:prop:genrealzerosfortbig1}.
This implies that there is an index $j\in \{1,\dots,N\}$ so that 
$|\be_j^N-x^N_i| < \frac12 m_N$. The discs
$B_i = \{z\in \C \mid |x^N_i-z| <\frac12 m_N\}$ are disjoint 
by construction 
and, since for any $i$ there exists a $j$ such that the zero $\be_j^N\in B_i$, we get a bijection $x^N_i \mapsto \be_j^N$ between the zeros of
$p_N$ and of $x\mapsto q_N(x;t)$. This forces $\be_j^N\in \R$ and 
$\be_j^N$ to be a simple zero of $q_N$.

Next we consider $x\mapsto q_{N+1}(x;t)$ and its zeros by $\be_j^{N+1}$, $1\leq j \leq N+1$. By the interlacing properties of the orthogonal polynomials (see, e.g., \cite[Ch.~1, \S~5]{Chih}, \cite{VanA}), we have 
\begin{equation*}
x_1^{N+1}<x_1^N < x_2^{N+1}<x_2^N < \cdots <x_{N}^{N+1} < x_N^N < x_{N+1}^{N+1}. 
\end{equation*}
We put $M_N = \min_{i,j} |x^N_i - x^{N+1}_j|>0$, so that in particular $M_N<m_N$ and $M_N<m_{N+1}$. Now we take $t\gg 0$ implying 
\begin{align*}
|q_{N-1}(x_i^N;t)| & < |t|^N\, |\al_N|\, |\lc(p_N)| \Bigl( \frac{M_N}{2} \Bigr)^N, \\
|q_{N}(x_j^{N+1};t)| & < |t|^{N+1}\, |\al_{N+1}|\, |\lc(p_{N+1})| \Bigl( \frac{M_N}{2} \Bigr)^{N+1}
\end{align*}
for all $i$ and all $j$, cf.~\eqref{eq:prop:genrealzerosfortbig1}. Using the result of the first part, we can order $\be_1^N<\be_2^N< \cdots <\be_N^N$ and 
$\be_1^{N+1}<\be_2^{N+1}<\cdots <\be_{N+1}^{N+1}$ so that 
\begin{align*}
\be_i^N\in (x_i^N-\frac12 M_N, x_i^N+\frac12 M_N), 
\qquad 
\be_j^{N+1}\in (x_j^{N+1}-\frac12 M_N, x_j^{N+1}+\frac12 M_N) 
\end{align*}
for all $i$ and all $j$. By definition of $M_N$, this gives 
\begin{equation*}
\be_1^{N+1}<\be_1^N<\be_2^{N+1}<\be_2^N<\dots < \be_N^{N+1} < \be_N^N <\be_{N+1}^{N+1}.
\qedhere
\end{equation*}
\end{proof}

Note that we also obtain an immediate corollary of the 
proof of Proposition \ref{prop:genrealzerosfortbig} on the interlacing of the zeros of the orthogonal polynomials and the partial sums. 

\begin{cor}\label{cor:prop:genrealzerosfortbig}
For $t\gg 0$  the real simple zeros of $x\mapsto q_N(x;t)$ and the real simple zeros of $p_{N+1}$ interlace. 
For $t\gg 0$  the real simple zeros of $x\mapsto q_{N+1}(x;t)$ and the real simple zeros of $p_{N}$ interlace. 
\end{cor}

\begin{proof} In the second part of the proof, we see that 
the interval $(x_i^N-\frac12 M_N, x_i^N+\frac12 M_N)$ containing
$\be_i^N$ is contained in the interval $(x_i^{N+1}, x_{i+1}^{N+1})$. Hence
\begin{equation*}
x_1^{N+1}<\be_1^N < x_2^{N+1}<\be_2^N < \cdots <x_{N}^{N+1} < \be_N^N < x_{N+1}^{N+1}. 
\end{equation*}
Similarly, the second part of the proof of Proposition~\ref{prop:genrealzerosfortbig} shows that 
\begin{equation*}
\be_1^{N+1}<x_1^N<\be_2^{N+1}<x_2^N <\cdots < \be_N^{N+1} < x_N^N <\be_{N+1}^{N+1}
\end{equation*}
implying the statement on  interlacing. 
\end{proof}

\begin{cor}\label{cor2:prop:genrealzerosfortbig}
For $t\gg 0$  the real simple zeros of $x\mapsto q_N(x;t)$ 
are contained in the convex hull of the orthogonality measure for the orthogonal polynomials $(p_n)_{n\in \N}$.
\end{cor}

\begin{proof} We know that this holds for the zeros of any of the polynomials $p_n$, in particular for $p_{N+1}$, hence 
Corollary \ref{cor:prop:genrealzerosfortbig} implies the statement. 
\end{proof}

Proposition \ref{prop:genrealzerosfortbig} deals with the properties  of the partial sums for $t\gg 0$. On the other hand, $q_N(x;t)$ tends to 
$1$, which has no zeros, as $t\to 0$. Fixing 
$N$, and taking $(t_r)_{r\in \N}$ a decreasing sequence of positive
numbers with $\lim_{r\to \infty} t_r=0$ we get polynomials
$f_r(x) = q_N(x;t_r)$. Then the series $(f_r)_{r\in\N}$ is a series of holomorphic functions that converges uniformly on compact sets to the function identically equal to $1$. Now take $R>0$, and let $K=\overline{B_R}$ be the closed ball of radius $R$. According to Hurwitz's Theorem, see e.g.\ \cite[Thm.~6.4.1]{Simo}, there exists
$M\in \N$ such that the zeros of $f_r$ are outside of $K$ 
for all $r\geq M$. So the zeros of the partial sums $x\mapsto q_N(x;t)$ tend to infinity as $t\to 0$. Experimentally, we observe that the zeros tend to infinity through the complex plane as $t\to 0$. 

\begin{conj}\label{conj:criticalvaluet} For each $N>0$, there exists a threshold $t_{\operatorname{crit}}=t_{\operatorname{crit}}(N)>0$ such that the polynomial
$x\mapsto q_N(x;t)$ has $N$ real simple zeros for $t>t_{\operatorname{crit}}$, while at least one pair of complex conjugate zeros for $0<t<t_{\operatorname{crit}}$. 
\end{conj}

Conjecture \ref{conj:criticalvaluet} suggests that there is only one critical value for $t$ for the change in behaviour of the zeros of $q_N$.
It would be of great interest to understand the behaviour of $t_{\operatorname{crit}}(N)$ as a function of~$N$.

\section{Examples}
\label{sec:examples}
In this section, we consider three examples: the partial sums of Hermite, Charlier and Lommel polynomials. The partial sums of Hermite polynomials represent the simplest case, and we describe some additional properties in this context. The Charlier polynomials provide an example with discrete orthogonality. Finally, the Lommel polynomials illustrate a family of polynomials that are not included in the Askey scheme. The plots provided in this section were generated using Python and Maple, and the codes are is available in a GitHub repository \cite{GitHub}.
Our principal theoretical sources for this part are the books~\cite{AndrAR}, \cite{Isma} and~\cite{KLS2010}.

\subsection{Partial sums of Hermite polynomials} 
\label{sec:hermite}
Our first and principal example comes from the weight $w(x)=e^{-x^2}$ on the interval $(-\infty, \infty)$. The corresponding Hermite polynomials and their generating function are given by
\[
H_n(x)=(-1)^n \left(\frac{d^n}{dx^n} w(x) \right) w(x)^{-1}, \quad \sum_{n=0}^\infty \frac{H_n(x)}{n!}t^n = e^{-t(t-2x)}.
\]
The three-term recurrence relation, the forward and backward shift operators are given by
\begin{equation}
\label{eq:Hermite-rec-shifts}
\begin{split}
xH_n(x) &= \frac12 H_{n+1}(x) + nH_{n-1}(x), \qquad \frac{d}{dx}H_n(x) = 2n H_{n-1}(x), \\
&\qquad \frac{d}{dx}H_n(x) -2x  H_n(x) = -H_{n+1}(x).
\end{split}
\end{equation}
We consider the partial sums of the Hermite polynomials with the following normalisation:
\begin{equation}
\label{eq:Hermite-partialsums}
    q_m(x;t) = \sum_{n=0}^m \frac{H_{n}(x) t^n}{n!}.
\end{equation}
Using Lemma \ref{lem:genfun} the generating function for the partial sums $q_m(x;t)$ is directly obtained:
\[\sum_{m=0}^\infty q_m(x;t)y^m = \frac{e^{-yt(yt - 2x)}}{1-y}.\]
The four-term recurrence relation for the partial sums follows directly from Lemma \ref{lem:gen4termrecurforpartialsums} and the recurrence relation of the Hermite polynomials \eqref{eq:Hermite-rec-shifts}. Taking into account that $\alpha_n = 1/n!$, $a_m=1/2$, $b_m=0$ and $c_m=m$, we get
\[
(m+1)q_{m+1}(x;t) -(2xt+(m+1)) q_m(x;t) +2 (t^2+tx) q_{m-1}(x;t) 
-2t^2q_{m-2}(x;t) = 0
\]
for $m\ge0$, with initial conditions as in Lemma~\ref{lem:gen4termrecurforpartialsums}.
In the following proposition we give a list of identities for the partial sums
$q_m(x;t)$ which follow from the properties of Hermite polynomials.

\begin{prop}
    \label{prop:qm-Hermite-identities}
    Let $q_m(x;t)$ be the partial sums of Hermite polynomials normalised as in \eqref{eq:Hermite-partialsums}. Then
    \begin{enumerate}
    \item[\textup{(1)}] The derivatives with respect to $t$ and $x$ are given by
    \begin{align*}
        \frac{\partial}{\partial x} q_m(x;t) &= 2tq_{m-1}(x;t), \\
        \frac{\partial}{\partial t} q_m(x;t) &= 2xq_{m-1}(x;t) -2t q_{m-2}(x;t).
    \end{align*}
    \item[\textup{(2)}] The partial sums are solutions to the partial differential equation
    $$t\,\frac{\partial}{\partial t} q_m(x;t) = x\, \frac{\partial}{\partial x} q_m(x;t) -\frac{1}{2}\, \frac{\partial^2}{\partial x^2} q_m(x;t)$$
    with the boundary condition $q_m(x;0) = 1$.
    \end{enumerate}
\end{prop}

\begin{figure}[t]
    \centering
    \begin{subfigure}[t]{0.325\textwidth}
        \centering
        \includegraphics[width=\textwidth]{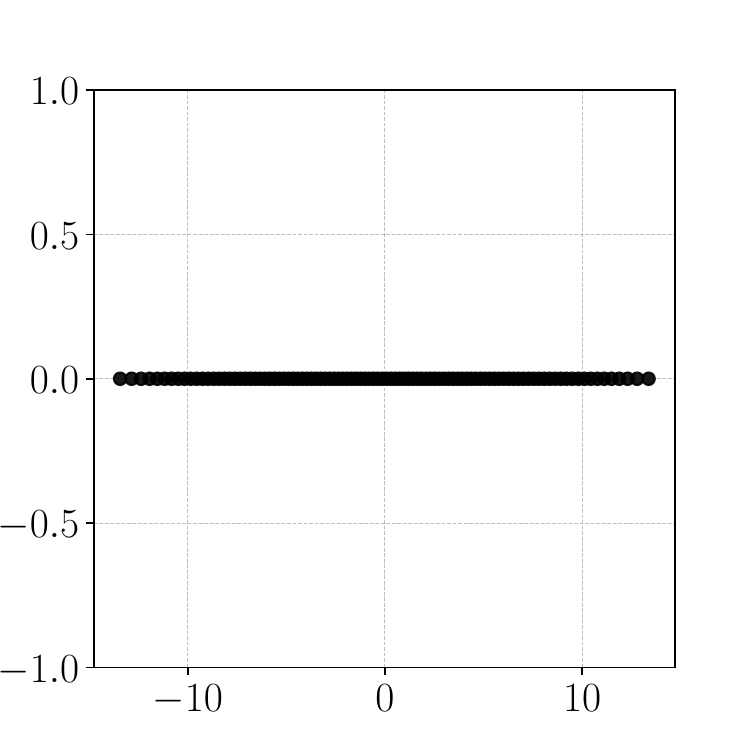}
    \end{subfigure}
    \hfill
    \begin{subfigure}[t]{0.325\textwidth}
        \centering
        \includegraphics[width=\textwidth]{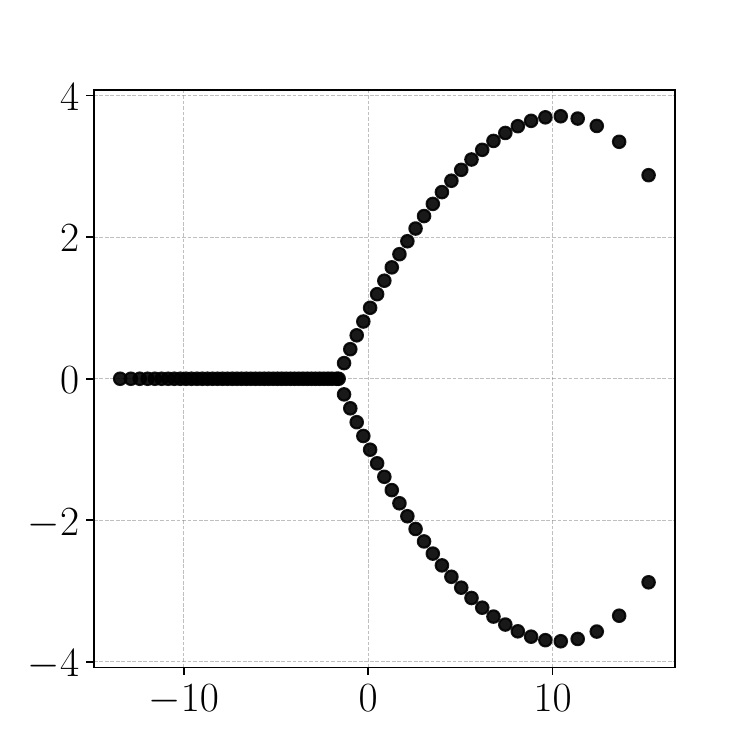}
    \end{subfigure}
    \begin{subfigure}[t]{0.325\textwidth}
        \centering
        \includegraphics[width=\textwidth]{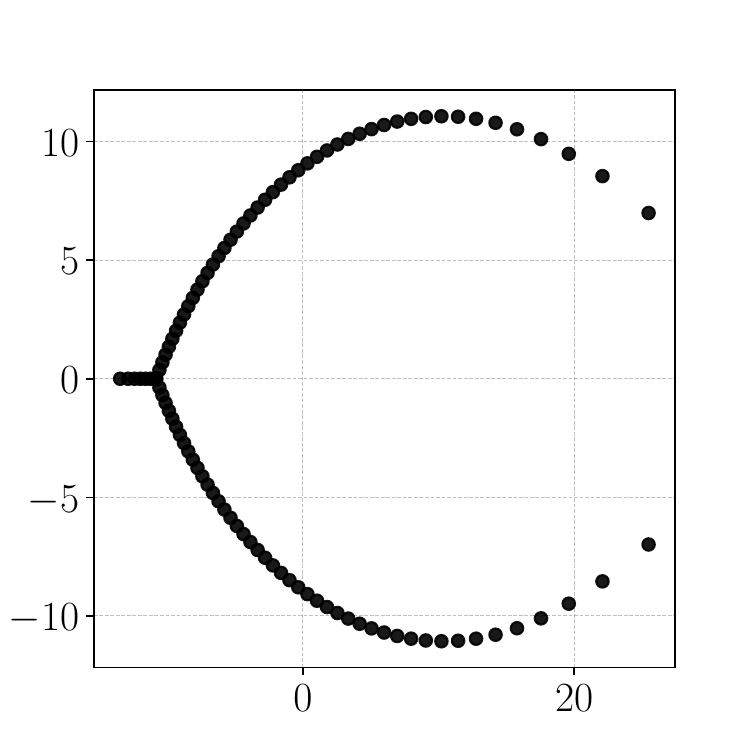}
    \end{subfigure}
    \caption{
        Zeros of the partial sums of Hermite polynomials $q_m(x;t)$ for $m=100$; $t=20$ (left), $t=3.48$ (center) and $t=1.69$ (right).
    }
    \label{fig:Hermite-zeros-3ts}
\end{figure}

Next we discuss the zeros of the partial sum polynomials. From Proposition \ref{prop:genrealzerosfortbig} we know that the zeros of $q_m(x;t)$ are real and simple for large $t$, see Figure \ref{fig:Hermite-zeros-3ts} (left). Moreover, in this regime, the zeros of $q_m(x;t)$ are close to the zeros of the $m$-th Hermite polynomial $H_m(x)$. Conjecture \ref{conj:criticalvaluet} is verified numerically for this example. We find it convenient to consider the rescaled partial sums
\begin{equation}
\label{eq:rescaled_hermite_partsums}
\widetilde q_m(x;t) = q(\sqrt{m}x;\sqrt{m}t).
\end{equation}
With this normalisation the zeros remain bounded for large $t$ and the critical $t$ seems to converge to a fixed value as $m\to \infty$. As $t$ approaches the critical value $t_{\textrm{crit}}$, the two largest zeros of $\widetilde q_m(x;t)$ get closer and finally collide at $t=t_{\textrm{crit}}$. For $t<t_{\textrm{crit}}$ these zeros move into the complex plane. The trajectories of these zeros and their collision are depicted in Figure~\ref{fig:HermiteCharlierN10}.

For the partial sums of Hermite polynomials, we  have a particular structure for the zeros. 

\begin{prop}
    \label{prop:double_zeros_qm_Hermite}
    Let $x_0,t_0\in \mathbb{R}$ be such that $q_m(x;t_0)$ has a double zero at $x=x_0$. Then $x_0$ is a zero of $H_m(x)$.
\end{prop}

\begin{proof}
The proof is a direct consequence of Proposition \ref{prop:qm-Hermite-identities}. Since $q_m(x;t_0)$ has a double zero at $x_0$, we have that $q_m(x_0;t_0) = 0$, and $q'_m(x_0;t_0) = 2t_0 q_{m-1}(x_0;t_0)=0$. Hence
$$0=q_m(x_0;t_0) - \frac{q'_m(x_0;t_0)}{2t_0} = \frac{t_0^{m}}{m!}H_m(x_0).$$
Since $t_0>0$, the proof is complete.
\end{proof}

Proposition \ref{prop:double_zeros_qm_Hermite} tells us that the double zeros of $\widetilde q_m(x;t)$ can only occur at a zero of the Hermite polynomial $H_m(\sqrt{m}x)$; this can be seen in Figure~\ref{fig:HermiteCharlierN10}. We observe that the trajectories of the zeros of $q_m(x;t)$ start at the zeros of $H_m(\sqrt{m}x)$ for large $t$ and the collisions occur exactly at half of the zeros of $H_m(\sqrt{m}x)$. This phenomenon does not seem to be present in other examples and it is only present in the case of partial sums of Hermite polynomials with the specific normalisation \eqref{eq:Hermite-partialsums}.

\begin{figure}[t]
    \centering
    \begin{subfigure}[t]{0.49\textwidth}
        \centering
        \includegraphics[width=\textwidth]{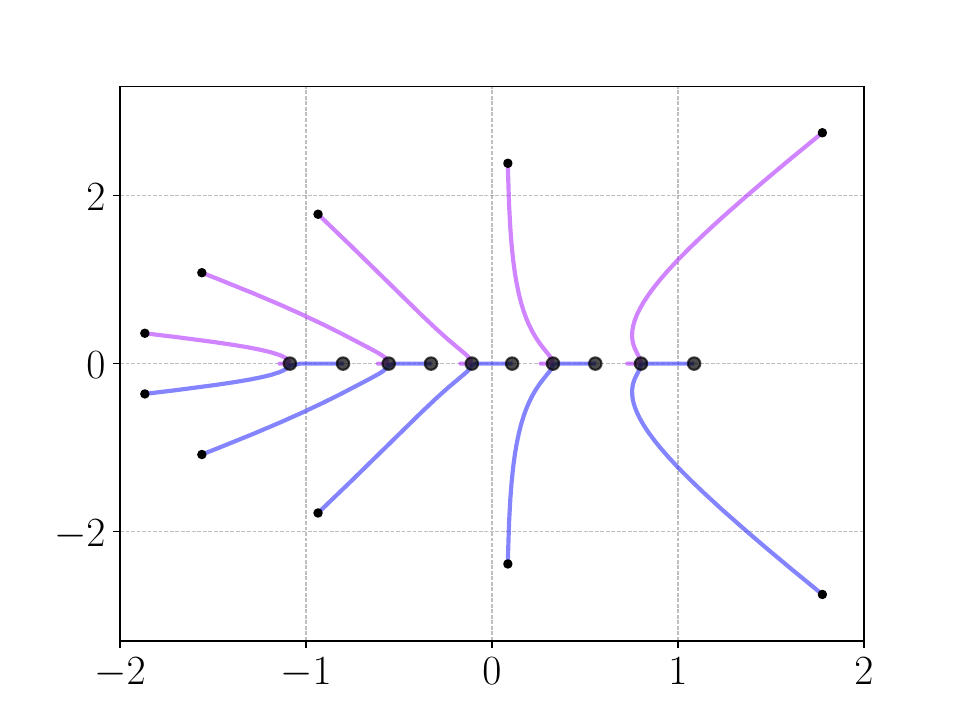}
    \end{subfigure}
    \hfill
    \begin{subfigure}[t]{0.49\textwidth}
        \centering
        \includegraphics[width=\textwidth]{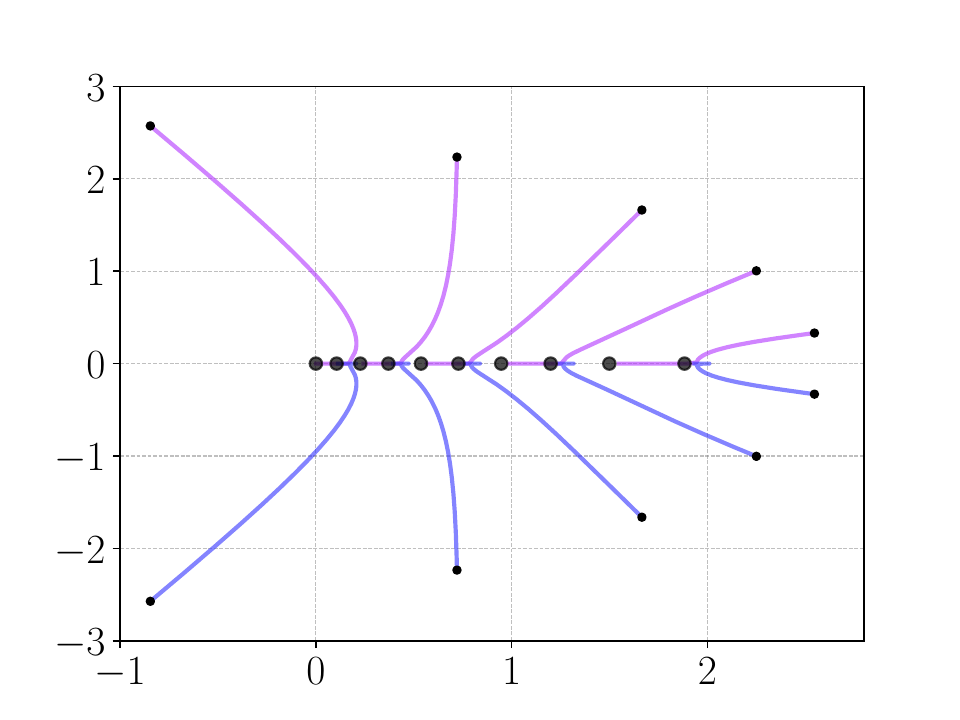}
    \end{subfigure}
    \caption{
        \textsl{Left}: Zeros of rescaled partial sums of Hermite polynomials. The small circles in the complex plane are the zeros of $\widetilde q_m(x;t_{\max})$. The blue and pink lines are the trajectories of these zeros in the range $t_{\min}$ to $t_{\max}$. The large circles on the real line are the zeros of the classical Hermite  polynomials $H_m(\sqrt{m} \, x)$. In this example: $m=10$, $t_{\max}=6$, $t_{\min}=0.1$.\\
        \textsl{Right}: Zeros of partial sums of rescaled Charlier polynomials $q^{(3)}_{m}(m x;t)$ from $t_{\min}$ to $t_{\max}$. The large circles on the real line are the zeros of the classical Charlier  polynomials $C_{m}^{(3)}(m x)$. In this example: $m=10$, $t_{\max}=6$, $t_{\min}=0.57$.}
        \label{fig:HermiteCharlierN10}
\end{figure}

For each $m\in \mathbb{N}_0$, the critical time $t_{\textrm{crit}}$ occurs when the first two zeros of the partial sums collide. Note that $\widetilde q_m(x;t)$ does not have double zeros for $t>t_{\textrm{crit}}$ and that, as a function of $x$, the polynomial $\widetilde q_m(x;t_{\textrm{crit}})$ has a double zero at a point $x_0$. By Proposition \ref{prop:double_zeros_qm_Hermite}, $x_0$ is a zero of the rescaled Hermite polynomial $H_m(\sqrt{m}x)$. If we fix $x=x_0$, the discriminant of $\widetilde q_m(x;t)$ has a zero at $t=t_{\textrm{crit}}$. Taking these into account, we can write a method to compute numerically $t_{\textrm{crit}}$:
\begin{enumerate}
    \item Compute the discriminant $\Delta(x;t)$ of $\widetilde q_m(x;t)$.
    \item For each of the zeros $x_0$ of $H_m(\sqrt{m}x)$, solve $\Delta(x_0;t)=0$ for the variable $t$.
    \item The largest solution of the previous step is precisely the critical time $t_{\textrm{crit}}$.
\end{enumerate}
Steps (2) and (3) can be simplified by observing that the first collision seems to occur at the the second largest zero of $H_m(\sqrt{m}x)$ so that we only need to solve $\Delta(x_0;t)=0$ for this case. In Table \ref{tab:critical-values-Hermite} we give the values of $t_{\textrm{crit}}$ for different degrees.

\begin{table}[ht]
    \centering
    \begin{tabular}{ccc}
        \toprule
        $m$ & $x_0$ & $t_{\textrm{crit}}$ \\
        \midrule
        10  & 0.800920079 & 0.6926318429 \\
        20  & 1.029414690 & 0.7190535658\\
        50  & 1.205301838 & 0.7334164664 \\
        100 & 1.282379975 & 0.7360578398 \\
        150 & 1.313478054 & 0.7358415083 \\
        \bottomrule
    \end{tabular}
    \captionsetup{format=plain, labelfont=bf, textfont=it}
    \caption{Critical values $t_{\textrm{crit}}$ for the rescaled partial sums of Hermite polynomials \eqref{eq:rescaled_hermite_partsums} and the second largest zero $x_0$ of the Hermite polynomial $H_m(\sqrt{m}x)$ for different degrees.}
    \label{tab:critical-values-Hermite}
\end{table}

\subsection{Partial sums of Charlier polynomials}
\label{sec:charlier}
Charlier polynomials $C^{(a)}_n(x)$ are orthogonal with respect to the weight function $e^{-a} a^x/x!$ on the non-negative integers, where $a>0$ is a parameter. Charlier polynomials and their generating function are given by:
\[
C^{(a)}_n(x) = \nabla^n \left( \frac{a^x}{x!} \right)\frac{x!}{a^x}, \qquad \sum_{n=0}^\infty C^{(a)}_n(x) \frac{t^n}{n!} = e^t \left(1 - \frac{t}{a}\right)^x.
\]
The three-term recurrence relation and the backward shift operator are:
$$xC^{(a)}_n(x) = -a C^{(a)}_{n+1}(x) + (n+a) C^{(a)}_{n}(x) - n C^{(a)}_{n-1}(x), \qquad \Delta_x C^{(a)}_n(x) = -\frac{n}{a} C^{(a)}_{n-1}(x).$$
We consider the following normalised partial sums.
\begin{equation}
\label{eq:Charlier-partialsums}
    q^{(a)}_m(x;t) = \sum_{n=0}^m \frac{C^{(a)}_{n}(x) t^n}{n!}.
\end{equation}
By Lemma \ref{lem:genfun}, the generating function of the Charlier partial sums is
\begin{equation}
\label{eq:Charlier-ps-generating}
\sum_{m=0}^\infty q^{(a)}_m(x;t)y^m = \frac{e^{ty}\left( 1-\frac{ty}{a}\right)^x }{1-y}.
\end{equation}
The four-term recurrence relation for the partial sums follows from Lemma \ref{lem:gen4termrecurforpartialsums} and the recurrence relation of the Charlier polynomials:
\begin{multline*}
-a(m+1) q^{(a)}_{m+1}(x;t) + \big(a(m + t + 1) + t(m - x)\big) q^{(a)}_{m}(x;t)
\\
-t(t + m + a - x) q^{(a)}_{m-1}(x;t) + t^2 q^{(a)}_{m-2}(x;t) = 0
\end{multline*}
for $m\ge0$, with initial conditions as in Lemma~\ref{lem:gen4termrecurforpartialsums}.
As in the case of the partial sums of Hermite polynomials, we have a simple lowering operator:
$$\Delta_x q_m^{(a)}(x;t) = -\frac{t}{a} q_{m-1}^{(a)}(x;t).$$
Proposition \ref{prop:double_zeros_qm_Hermite} is no longer valid for Charlier polynomials. One of the consequences of this fact is that the double zeros of $q_m^{(a)}(x;t)$ do not occur at a zero of a Charlier polynomial. This can be observed in Figure \ref{fig:HermiteCharlierN10}. The blue trajectories start at the even zeros (ordered from smallest to largest) of the Charlier polynomial $C^{(a)}_{10}(x)$, the pink trajectories start at the odd zeros. In contrast to the case of the partial sums of Hermite polynomials, the trajectories of the zeros of the partial sums of Charlier polynomials do not intersect at a zero of a Charlier polynomial.

\subsection{Partial sums of Lommel polynomials}
\label{sec:lommel}

\begin{figure}[t]
    \centering
    \begin{subfigure}[t]{0.325\textwidth}
        \centering
        \includegraphics[width=\textwidth]{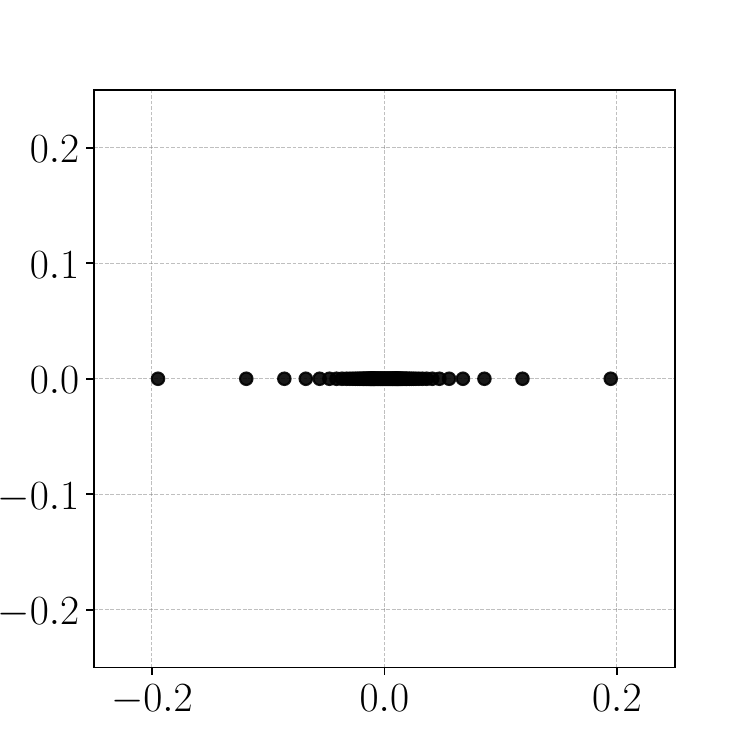}
    \end{subfigure}
    \hfill
    \begin{subfigure}[t]{0.325\textwidth}
        \centering
        \includegraphics[width=\textwidth]{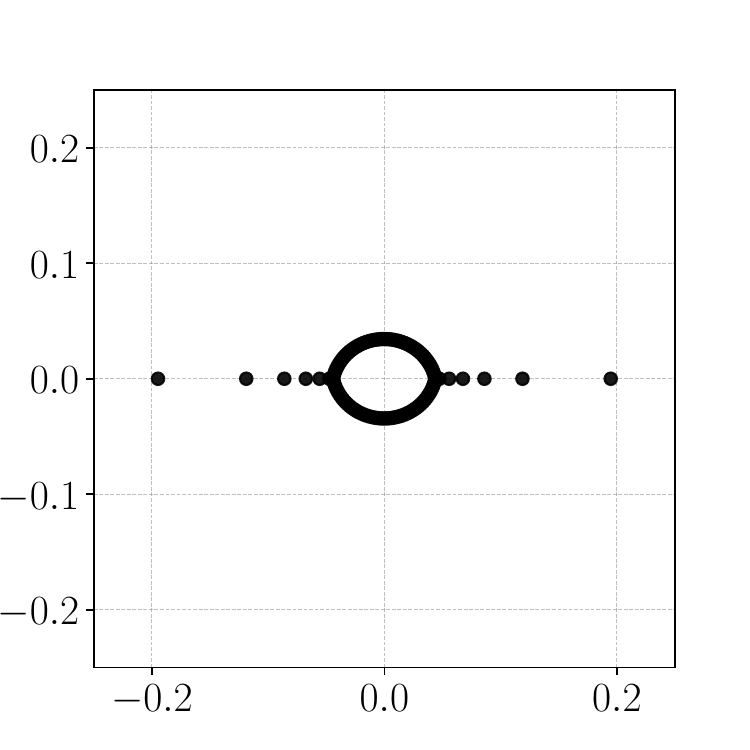}
    \end{subfigure}
    \begin{subfigure}[t]{0.325\textwidth}
        \centering
        \includegraphics[width=\textwidth]{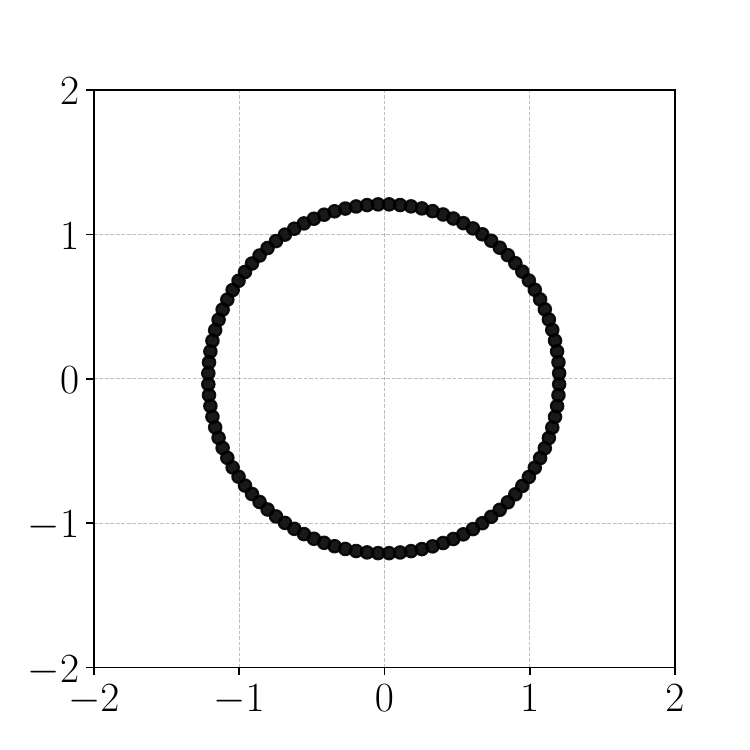}
    \end{subfigure}
    \caption{
        Zeros of the partial sums of Lommel polynomials $q^{(3)}_m(x;t)$ for $m=100$; $t=1.8$ (left), $t=0.286$ (center) and $t=0.01$ (right).
    }
    \label{fig:Lommel-zeros-3ts}
\end{figure}

The Lommel polynomials $h_{m+1,\nu}(x)$ are a class of orthogonal polynomials which depend on one parameter $\nu>0$. Unlike Hermite and Charlier polynomials, Lommel polynomials are not part of the Askey scheme.

Lommel polynomials and their generating function are given by
\[
h_{m+1,\nu}(x)  = \sum_{n=0}^{\left\lfloor m/2 \right\rfloor} \frac{(-1)^n (m - n)! \,\Gamma(\nu + m - n)}{n!\,(m - 2n)!\,\Gamma(\nu + n)} \left( \frac{1}{2x} \right)^{2n - m},
\]
while the three-term recurrence relation for them is
\[
h_{m+1,\nu}(x) = 2x(m + \nu) h_{m,\nu}(x) - h_{m-1,\nu}(x), \qquad h_{-1,\nu}(x) = 0, \quad h_{0,\nu}(x) = 1.
\]
We consider the following partial sums of Lommel polynomials:
\begin{equation}
\label{eq:Lommel-partialsums}
    q^{(\nu)}_m(x;t) = \sum_{n=0}^m h_{m,\nu}(x) t^n.
\end{equation}
The four-term recurrence relation for the partial sums follows from Lemma \ref{lem:gen4termrecurforpartialsums} and the recurrence relation of the Lommel polynomials:
\begin{equation*}
    \label{eq:4term-Lommel}
      q^{(\nu)}_{m+1}(x;t) -\big(1 + 2tx(m + \nu)\big) q^{(\nu)}_{m}(x;t) + 2t\Big(\frac{t}{2} + x(m + \nu)\Big) q^{(\nu)}_{m-1}(x;t) -t^2 q^{(\nu)}_{m-2}(x;t) = 0
\end{equation*}
for $m\ge0$, with initial conditions as in Lemma~\ref{lem:gen4termrecurforpartialsums}.
As in the case of the partial sums of  Charlier polynomials the double zeros of \( q_m^{(\nu)}(x;t) \) do not occur at zeros of a Lommel polynomial.

\subsection{The limit as $t\to 0$ and the connection with the Szeg\H{o} curve}
\label{sec:szego}

\begin{figure}[t]
    \centering
    \begin{subfigure}[t]{0.49\textwidth}
        \centering
        \includegraphics[width=\textwidth]{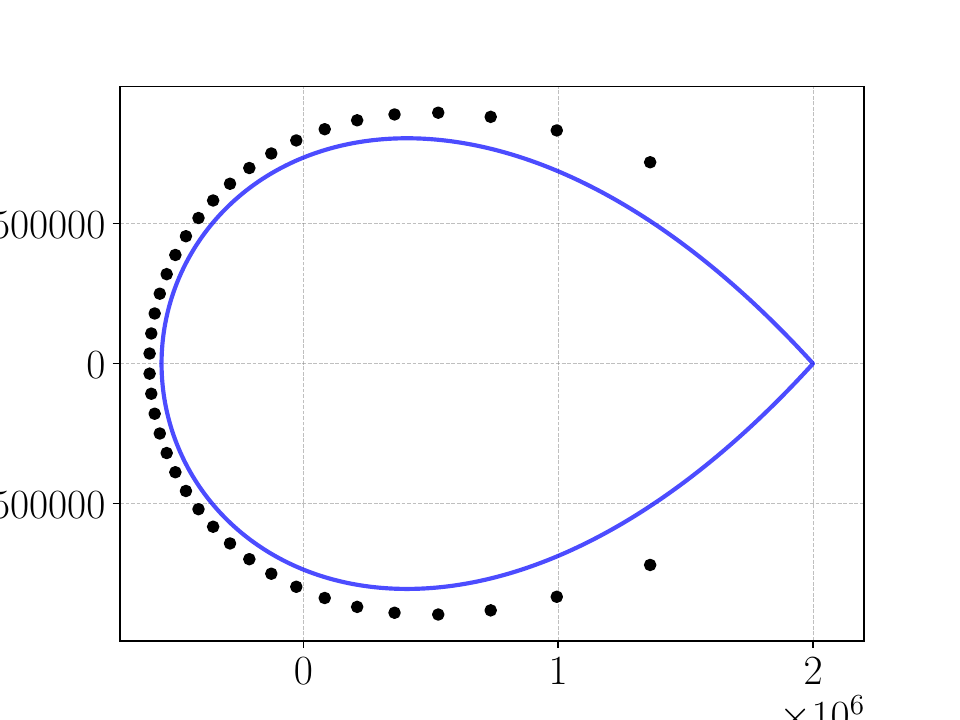}
    \end{subfigure}
    \hfill
    \begin{subfigure}[t]{0.49\textwidth}
        \centering
        \includegraphics[width=\textwidth]{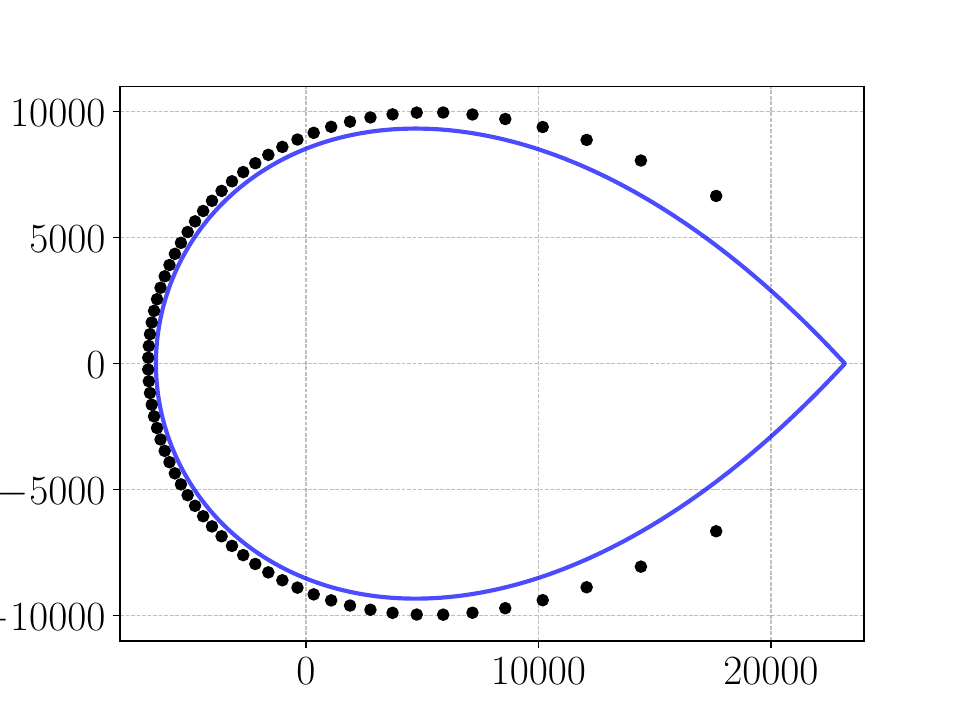}
    \end{subfigure}
    \caption{
        Zeros of the partial sums of Hermite polynomials $q_m(x;t)$ and the scaled Szeg\H{o} curve $|ze^{1-z}|=1$ where $z=2xt/m$ for $t=0.00001$ and $m=40$ (left), $m=70$ (right).
    }
    \label{fig:HermiteSzegoN40N70}
\end{figure}

With the help of Lemma \ref{lem:gen4termrecurforpartialsums} we can write the generating function for the partial sums of the exponential function,
\[
S_m(t)=\sum_{n=0}^m\frac{t^n}{n!}.
\]
The result is
\[
\sum_{m=0}^\infty S_m(t)y^m=\frac{e^{ty}}{1-y}.
\]
On the other hand, for the generating function of the partial sums \eqref{eq:Hermite-partialsums} of the Hermite polynomials we get
\[
\sum_{m=0}^\infty q_m(x;t)y^m
= \frac{e^{-yt(yt - 2x)}}{1-y}
=\frac{e^{-y^2t^2+2xyt}}{1-y}
=e^{-y^2t^2}\sum_{l=0}^\infty S_l(2xt)y^l.
\]
Comparing the terms in the $y$-expansions on the two sides implies
\[
q_m(x;t)
=\sum_{k=0}^{\lfloor m/2\rfloor}\frac{(-1)^kt^{2k}}{k!}\,S_{m-2k}(2xt).
\]
When $t>0$ is chosen sufficiently small (but fixed), we have $q_m(x;t)=S_m(2xt)+O(t^2)$, so that the zeroes of $q_m(x;t)$ are close to the corresponding zeroes of $S_m(2xt)$.
The latter are distributed over the Szeg\H{o} curve $|ze^{1-z}|=1$ where $z=2xt/m$;
thus, if $t\to0^+$ and $x_1^m(t),\dots,x_m^m(t)$ denote the zeros of $q_m(x;t)$, then $z_j=2tx_j^m(t)/m$ accumulate on the closed loop of the curve $|ze^{1-z}|=1$. This is illustrated in Figure \ref{fig:HermiteSzegoN40N70}.

\begin{figure}[t]
    \centering
    \begin{subfigure}[t]{0.49\textwidth}
        \centering
        \includegraphics[width=\textwidth]{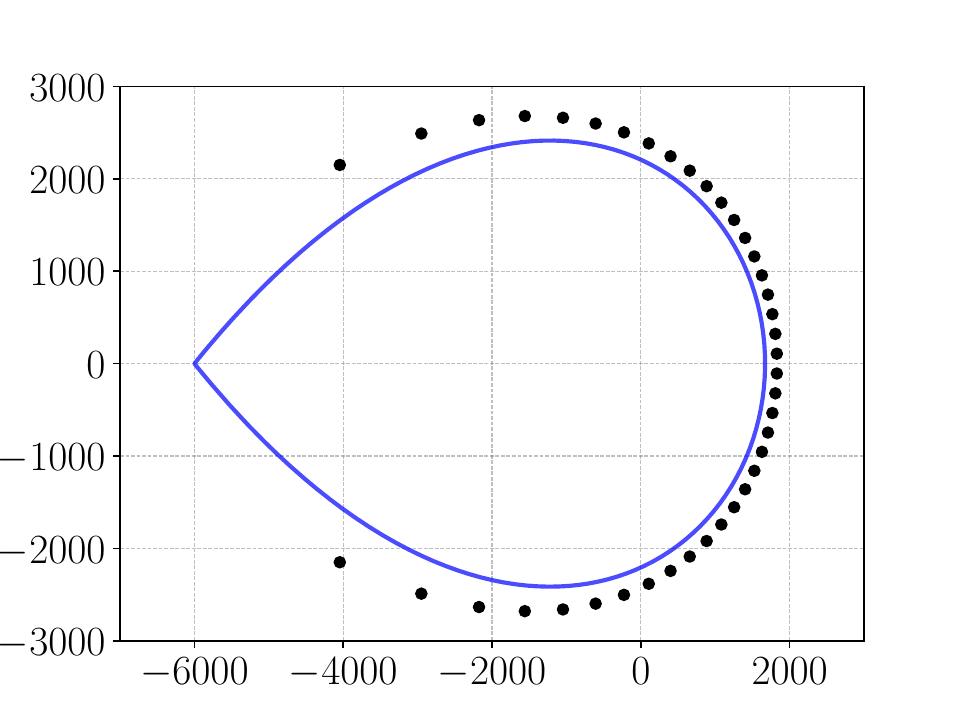}
    \end{subfigure}
    \hfill
    \begin{subfigure}[t]{0.49\textwidth}
        \centering
        \includegraphics[width=\textwidth]{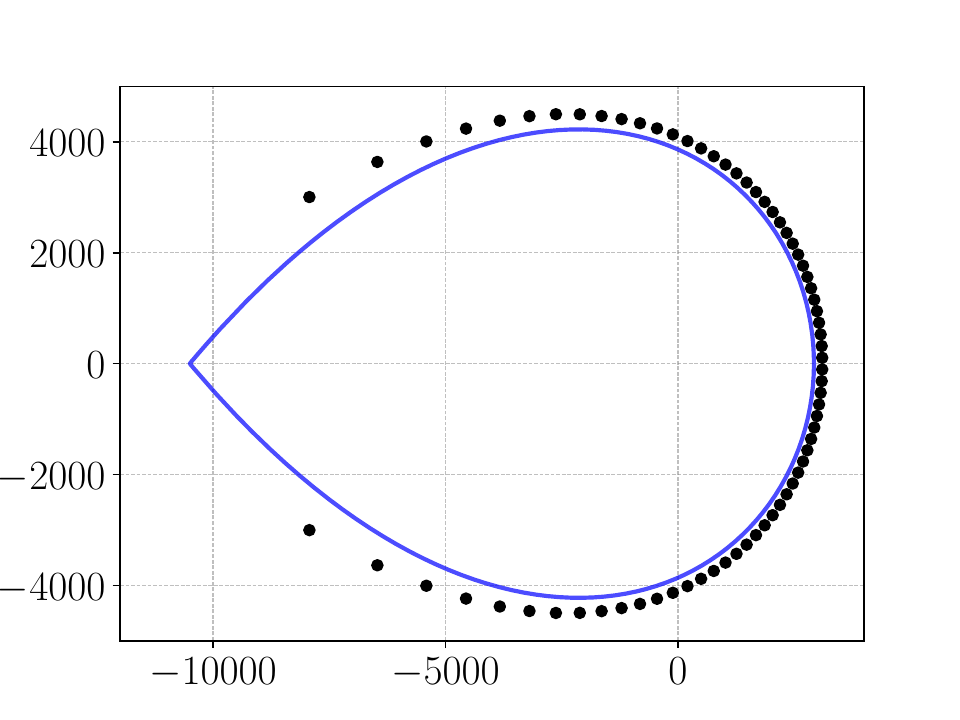}
    \end{subfigure}
    \caption{
        Zeros of the partial sums of Charlier polynomials $q^{(a)}_m(x;t)$ and the scaled Szeg\H{o} curve $|ze^{1-z}|=1$ where $z=-xt/am$ for $a=3$, $t=0.000001$ and $m=40$ (left), $m=70$ (right).
    }
    \label{fig:CharlierN40N70}
\end{figure}

\medskip
The phenomenon of such a zero distribution along the Szeg\H{o} curve pertains, after the choice of a suitable regime of parameters' dependence, for the partial sums of other families of orthogonal polynomials whose generating functions are intrinsically linked with the exponential function. This is, for example, the case for partial sums of  Charlier polynomials treated in Subsection~\ref{sec:charlier}. From the generating function of the partial sums \eqref{eq:Charlier-ps-generating} of Charlier polynomials we have
\begin{align*}
\sum_{m=0}^\infty q^{(a)}_m(x;t) y^m &= \frac{e^{ty}\left( 1-\frac{ty}{a} \right)^x }{1-y} = 
 \frac{e^{ty}\left( 1-\frac{ty}{a} \right)^x }{1-y} =  \frac{e^{ty} }{1-y} \exp \left(-\sum_{n=1}^\infty \frac{xt^ny^n}{na^n} \right) \\
 & \sim \frac{e^{ty} }{1-y} e^{-\frac{xty}{a} } = e^{ty} \sum_{\ell=0}^\infty S_\ell\left(-\frac{xt}{a}\right)y^\ell.
\end{align*}
In the second line of the above equation, we approximate the infinite series by considering only the first term. From these considerations, similarly to the case of the partial sums of Hermite polynomials, we find that $q^{(a)}_m(x;t) \sim S_m\left(-\frac{xt}{a}\right) + O(t)$. Hence, if $t\to0^+$ and $x_1^m(t),\dots,x_m^m(t)$ denote the zeros of $q^{(a)}_m(x;t)$, we expect that $z_j=-tx_j^m(t)/am$ accumulate on the closed loop of the curve $|ze^{1-z}|=1$. In Figure \ref{fig:CharlierN40N70} we numerically verify this behavior for a particular choice of the parameter~$a$.

\section{Conclusive remarks}
\label{sec:5}

Originally we came across very particular partial sums of Gegenbauer polynomials in our study of the zeros of entries of matrix-valued orthogonal polynomials \cite{KRZ24}. The evidence of special distribution of those zeros prompted us to look in the literature for such situations; unfortunately, not so much in this direction is ever discussed\,---\,the paper \cite{TranZ} is a rare example.

At the same time, our partial sums of Hermite polynomials from Subsection~\ref{sec:hermite} share numerous similarities with the so-called heat polynomials \cite{RW58}; these partial sums also show up in the context of the empirical distribution of zeros of an algebraic polynomial \cite{HHJK24}.
Therefore, making the intertwining connections explicit may help understanding related structures of discrete and mixed continuous-discrete analogs of the heat equation.

It seems to us that partial sums of orthogonal polynomials have a life of their own, and many of its episodes deserve careful attention.


\end{document}